\documentclass[11pt,reqno]{amsart}

\usepackage[T1]{fontenc}
\usepackage[utf8]{inputenc}
\usepackage{amsmath,amssymb,amsfonts,amsthm,mathtools,bbm,mathrsfs}
\usepackage{array,booktabs,tabularx}
\usepackage{graphicx}
\usepackage{xcolor}
\usepackage{microtype}
\usepackage[hidelinks]{hyperref}
\allowdisplaybreaks
\numberwithin{equation}{section}

\newtheorem{theorem}{Theorem}[section]
\newtheorem{corollary}[theorem]{Corollary}
\newtheorem{proposition}[theorem]{Proposition}

\theoremstyle{definition}
\newtheorem{definition}[theorem]{Definition}
\newtheorem{example}[theorem]{Example}
\theoremstyle{remark}

\newcommand\mN{{\mathbb N}}
\newcommand\sA{\mathsf{A}}
\newcommand\bA{{\mathscr A}}
\newcommand\cE{{\mathcal E}}

\newcommand\bF{{\mathbf F}}
\newcommand\mI[1]{\mathbbm{1}_{#1}}
\newcommand\md{\,\mathrm d}
\newcommand\Ch{\operatorname{Ch}}

\newcommand\id{\operatorname{id}}

\renewcommand\le{\leqslant}
\renewcommand\ge{\geqslant}

\title[Conditional Aggregation in Context-Aware Fuzzy Systems]{Conditional Aggregation in Context-Aware Fuzzy Systems: A Fibered Representation Theory}

\author{Ondrej Hutn\'{i}k}
\thanks{This work has been submitted to the IEEE for possible publication. Copyright may be transferred without notice, after which this version may no longer be accessible.}
\address{Institute of Mathematics, Faculty of Science, Pavol Jozef \v{S}af\'{a}rik University in Ko\v{s}ice, Jesenn\'{a} 5, 040 01 Ko\v{s}ice, Slovakia}
\email{ondrej.hutnik@upjs.sk}

\author{Nat\'{a}lia Pu\v{s}k\'{a}rov\'{a}}
\address{Institute of Mathematics, Faculty of Science, Pavol Jozef \v{S}af\'{a}rik University in Ko\v{s}ice, Jesenn\'{a} 5, 040 01 Ko\v{s}ice, Slovakia}

\keywords{conditional aggregation operators, context-aware fuzzy systems, information fusion, capacity extension, universal integrals}
\hypersetup{
 pdftitle={Conditional Aggregation in Context-Aware Fuzzy Systems: A Fibered Representation Theory},
 pdfauthor={Ondrej Hutnik and Natalia Puskarova}
}

\begin{document}
\begin{abstract}
Conditional aggregation operators evaluate profiles relative to admissible contexts whose active sources, criteria, or rules may vary. For a fixed context, such an operator is indeed an ordinary monotone functional on the restricted profile. The nontrivial structure appears at the family level, where local functionals on different profile spaces must be compared and possibly represented by one global rule. We develop a fibered representation theory for families of conditional aggregation operators. After separating fiberwise properties from cross-context compatibility, we characterize global representability without assuming that the full universe is itself an admissible context. Such a family is generated by a monotone global functional if and only if its local values respect the pointwise order after canonical zero extension; when this holds, explicit least and greatest global generators are available. This yields a hierarchy whose inclusions can be strict between arbitrary, zero-extension-consistent, and globally representable families. For finite systems, we then establish an extremal extension theorem for partial capacities and derive an exact overlap criterion for gluing local capacities into global ones. Finally, calibrated representation-preserving integral schemes unify the Choquet, Sugeno, and Shilkret cases: zero-extension consistency is equivalent to projective compatibility of local set functions, while overlap-compatible local capacities yield sharp lower and upper fuzzy-integral scores. Evidence-fusion and multi-criteria applications interpret these results as robustness under inactive sources, context-dependent interaction, and structural uncertainty under incomplete cross-context information.
\end{abstract}

\maketitle

\section{Introduction}

Aggregation is a fundamental component of fuzzy systems, information fusion, decision making, and multi-criteria evaluation. Most classical models are formulated on a fixed universe of inputs: one operator combines a prescribed collection of criteria, rules, features, experts, or information sources; see
\cite{Grabisch_book2}. In many context-aware systems, however, the active collection itself varies. Sensors may become unavailable, rules may be activated only under specific operating conditions, criteria may be relevant only in selected decision scenarios, and sources may be admissible only in particular contexts. The object to be evaluated is then not merely a fixed-dimensional vector, but a profile together with the context on which that profile is active.

Conditional aggregation operators (CAOs), introduced in~\cite{BHHK} and subsequently used in generalized level-measure constructions~\cite{BHKK}, formalize aggregation relative to a conditional set. For one fixed conditional set, a CAO is equivalent to an ordinary monotone functional acting on the restricted profile. This observation is correct, but it does not capture the additional structure that appears when the conditional set varies. A family of conditional aggregation operators (FCA) consists of local aggregation rules defined on different profile spaces, and these rules may themselves depend on their contexts. The central issue is therefore not only how each local profile is aggregated, but also how the local rules are related when the active context changes.

This issue is connected with several established directions in fuzzy-systems research. Nonadditive aggregation and fuzzy integrals model interaction, redundancy, and complementarity among sources or criteria. The universal-integral framework provides a common setting for prominent nonadditive integrals~\cite{KMP2010}, while interval-valued Choquet aggregation, fuzzy-measure optimization, and computational generation of fuzzy measures remain active topics~\cite{Bustince2013,BeliakovWu2024,BeliakovSampling2024}. Recent contributions include Choquet-like integrals learned from data for classification and fault diagnosis~\cite{WangZhangShen2024}, Choquet-based multi-criteria procedures under richer fuzzy information~\cite{Qin2024}, and integral models for hesitant fuzzy information fusion~\cite{Gao2023}. These works confirm the continuing importance of nonlinear aggregation and interaction modeling in contemporary fuzzy systems.

A related line addresses incomplete, modular, or distributed information, including multi-criteria models with incompletely specified weights~\cite{WangZhang2013}, fuzzy group decision making with incomplete preferences~\cite{Capuano2018}, and fuzzy-network architectures combining local subsystems~\cite{Yaakob2017}. Such approaches usually retain a fixed ambient collection of variables and estimate, optimize, or complete unknown values or parameters. Our setting is complementary: the active domain itself changes, and different admissible domains may carry different aggregation laws. The question is therefore not how to complete one model on a fixed universe, but when local models defined on different domains are mutually compatible and when they admit one global representation.

We address this question by representing a family of conditional aggregation operators as a fibered system. Each admissible context $E$ carries a profile space $\bF_E$ and a monotone local functional $G_E$. An FCA is represented by the indexed family $\mathbf G=\{G_E:E\in\cE^0\}$. This viewpoint separates two logically distinct levels. Fiberwise properties, such as profile monotonicity, idempotency, homogeneity, or internality, describe individual local aggregation rules. Cross-context properties, such as monotonicity with respect to conditional sets, consistency under zero extension, equivariance, and global representability, describe relations between different fibers. Strong regularity of every local functional does not, by itself, guarantee coherent behavior when the context changes.

The first main contribution is a global representation theorem. We introduce \emph{global order compatibility}, which compares local profiles after their canonical zero extensions to the ambient space. This condition is necessary and sufficient for the existence of a nondecreasing global functional generating all local rules, even when the full universe is not an admissible context. Whenever such a representation exists, we construct its least and greatest monotone generators, thereby describing its nonuniqueness and characterizing uniqueness. The theorem also yields a hierarchy between globally representable, zero-extension-consistent, and arbitrary FCAs, with both inclusions strict for suitable context systems.

The second contribution is an extension theory for finite capacity-based systems. Local fuzzy measures specify a set function only on part of the Boolean lattice of the ambient universe. We prove that every monotone partial capacity containing the bottom and top elements admits normalized capacity extensions and derive explicit least and greatest extensions. For local capacities on overlapping contexts, the exact gluing condition is agreement on every common coalition. Thus, nested compatibility, which governs zero-extension consistency, and overlap compatibility, which governs global representability, are genuinely different requirements.

The third contribution connects this extension theory with fuzzy integrals. We introduce calibrated representation-preserving integral schemes, characterized by invariance under simultaneous zero extension of a profile and compatible extension of its monotone set function. The class contains the discrete Choquet integral and the smallest semicopula-based universal integrals, including the Sugeno and Shilkret integrals. For all such schemes, zero-extension consistency is equivalent to projective compatibility of the local set functions, while a global integral representation exists exactly under overlap compatibility. If the scheme is monotone with respect to the capacity, the extremal capacity extensions yield attainable lower and upper scores and hence the smallest interval containing all globally coherent evaluations.

These intervals are related to, but different from, interval-valued and robust fuzzy integrals, where uncertainty is already encoded in interval inputs or capacities~\cite{Bustince2013,GrecoRindone2013}. Here uncertainty is generated by the context structure: local capacities may be completely known, while coalitions involving sources or criteria that never occur together remain unspecified. The resulting bounds quantify structural uncertainty caused specifically by missing cross-context interaction information.

The proposed theory is therefore not another isolated aggregation operator or a competing axiomatization of fuzzy integrals. It provides a representation and extension layer for fuzzy systems with varying active sources, rules, or criteria. Weighted, Choquet, Sugeno, Shilkret, t-norm, t-conorm, or evidence-fusion mechanisms may serve as local rules; the fibered framework determines whether they form a coherent family, whether they arise from one global model, and what can be concluded when the global interaction structure is only partially specified. The applications illustrate robustness under inactive zero-valued sources, genuinely context-dependent interactions, and sharp multi-criteria score bounds without arbitrary completion of unassessed coalitions.

The paper is organized as follows. Section~\ref{sec:II} introduces CAOs and their fibered representation. Section~\ref{sec:III} distinguishes fiberwise and cross-context properties. Section~\ref{sec:IV} establishes the global representation theorem and hierarchy. Section~\ref{sec:V} develops the extension theory for partial and local capacities. Section~\ref{sec:VI} studies representation-preserving integral schemes and robust bounds. Section~\ref{sec:VII} classifies representative constructions, and Section~\ref{sec:VIII} presents the applications.

\section{Conditional Aggregation as a Fibered Object}\label{sec:II}

\subsection{Preliminaries}
Let $X$ be a nonempty set. When measurability is relevant, let $\Sigma\subseteq 2^X$ be an ambient $\sigma$-algebra. Let $\cE\subseteq\Sigma$ be a nonempty collection of admissible conditional sets and put $\cE^0:=\cE\setminus\{\emptyset\}.$ In this paper, any conditional aggregation operator is considered only for sets in $\cE^0$. The empty set is not part of the family-level theory below; it may be treated separately in applications where a specific convention is needed.

Denote by $\bF$ the class of all nonnegative bounded $\Sigma$-measurable functions on $X$. For $f,g\in\bF$, we write $f\le g$ on a set $E$ if $f(x)\le g(x)$ for every $x\in E$. The indicator of $E\in\Sigma$ is denoted by $\mI{E}$, i.e., $\mI{E}(x)=1$ for $x\in E$, and $\mI{E}(x)=0$ otherwise. When an indicator is used on a local fiber, it is understood as the corresponding restricted indicator.

\begin{definition}
Let $E\in\cE^0$. A mapping $\sA(\cdot|E)\colon\bF\to[0,\infty]
$ is called a \textit{conditional aggregation operator with respect to $E$} (CAO, for short) if it satisfies:
\begin{itemize}
\item[(C1)] $\sA(f|E)\le \sA(g|E)$ whenever $f,g\in\bF$ and $f\le g$ on $E$;
\item[(C2)] $\sA(\mI{E^c}|E)=0$, where $E^c=X\setminus E$.
\end{itemize}
\end{definition}

A \emph{family of conditional aggregation operators} over $\cE$ (FCA, for short) is a family
$$\bA=\{\sA(\cdot|E):E\in\cE^0\},$$
where each $\sA(\cdot|E)$ is a CAO with respect to $E$. For $E\in\cE^0$, let $\Sigma_E:=\{B\cap E:B\in\Sigma\}$
be the trace $\sigma$-algebra and let $\bF_E$ denote the class of all nonnegative bounded $\Sigma_E$-measurable functions on $E$. Equivalently, $\bF_E=\{f|_E:f\in\bF\}$, since every bounded $\Sigma_E$-measurable profile admits a bounded $\Sigma$-measurable extension to $X$.

\subsection{Restriction Principle}
Although a conditional aggregation operator is formally defined on the ambient function space $\bF$, its value depends exclusively on the information carried by the conditional set $E$. This suggests that the essential object is not the mapping $\sA(\cdot|E)\colon\bF\to[0,\infty],$
but rather the corresponding functional acting directly on the local profile space $\bF_E$. The next proposition makes this observation precise. It shows that every CAO is uniquely represented by a nondecreasing functional on $\bF_E$, which will later serve as the basic building block of the fibered framework.

\begin{proposition}[Restriction principle]\label{prop:restriction-principle}
Let $E\in\cE^0$. If $\sA(\cdot|E)$ is a CAO, then there is a unique mapping $G_E\colon\bF_E\to[0,\infty]$ such that
\begin{equation}\label{eq:restriction-representation}
\sA(f|E)=G_E(f|_E),\qquad f\in\bF.
\end{equation}
Moreover, $G_E$ is nondecreasing with respect to the pointwise order on $\bF_E$ and satisfies $G_E(0_E)=0$. Conversely, every nondecreasing $G_E\colon\bF_E\to[0,\infty]$ satisfying $G_E(0_E)=0$ generates a CAO by \eqref{eq:restriction-representation}.
\end{proposition}

\begin{proof}
If $f|_E=g|_E$, then $f\le g$ and $g\le f$ on $E$. By (C1),
$\sA(f|E)=\sA(g|E).$ Hence, $\sA(f|E)$ depends only on $f|_E$. Define
$G_E(u):=\sA(f|E),$ where $f\in\bF$ is any extension of $u\in\bF_E$. The previous argument shows that $G_E$ is well defined. 

To prove uniqueness, suppose that
$H_E\colon\bF_E\to[0,\infty]$ also satisfies $\sA(f|E)=H_E(f|_E)$
for every $f\in\bF$. Let $u\in\bF_E$, and let $f$ be any extension of $u$. Then $H_E(u)=\sA(f|E)=G_E(u),$ so $H_E=G_E$. Thus, $G_E$ is unique.

Monotonicity follows from (C1), and
$G_E(0_E)=0$ follows from (C2), since $\mI{E^c}|_E=0_E.$ The converse is immediate.
\end{proof}

The Restriction principle identifies the intrinsic local object associated with a conditional aggregation operator. For a fixed conditional set $E$, a CAO is completely characterized by the corresponding monotone zero-preserving functional
$G_E\colon\bF_E\to[0,\infty].$ The mathematical novelty of the conditional aggregation operator framework, however, does not lie in this local representation itself. It arises when the conditional set is allowed to vary and one studies a family of local functionals together with the compatibility relations between them. This viewpoint naturally leads to the fibered framework developed in the following sections.

From now on, we regard each $u\in\bF_E$ as an $E$-indexed profile. The term \emph{profile} emphasizes that each value is attached to a specific element of the conditional set rather than to a position in an ordered tuple. Consequently, no ordering of $E$ is required to define $G_E(u)$. If a coordinate representation is desired, one may choose an arbitrary bijection $\tau\colon I\to E$ and represent the same profile on the index set $I$ by the pullback $u\circ\tau$. Such a representation is merely a choice of coordinates. Different bijections produce different coordinate descriptions of the same profile, and independence of this choice is an additional invariance property rather than part of the definition. In the finite setting, this invariance reduces to the classical symmetry of aggregation functions.

\subsection{FCA Representation}
By Proposition~\ref{prop:restriction-principle}, every FCA can be represented by a family
$$\mathbf G=\{G_E:E\in\cE^0\},\qquad G_E\colon\bF_E\to[0,\infty].$$
The full family may also be encoded as a single mapping on the tagged disjoint union $$\mathfrak F_{\cE}:=\bigsqcup_{E\in\cE^0}(\{E\}\times\bF_E)$$ by setting $\mathbb G(E,u)=G_E(u)$. This is only a global representation of an indexed family. It is not a single CAO, since the argument contains both the conditional set $E$ and an $E$-indexed profile $u$.

\section{Fiberwise and Cross-Set Properties}\label{sec:III}

The family representation reveals two logically distinct levels of properties. A property is \emph{fiberwise} if it can be formulated for a single functional $G_E\colon\bF_E\to[0,\infty]$ without reference to any other conditional set. Monotonicity in the profile, idempotency, homogeneity, and symmetry under self-bijections of $E$ are properties of this type. For example, a CAO is idempotent on a fixed conditional set $E$ if
\begin{equation}\label{eq:CAOidempotent}
G_E(b 1_{E})=b,\qquad b\ge0.
\end{equation}
The FCA is fiberwise idempotent if this holds for every $E\in\cE^0$. If the condition is imposed only for sets in a subcollection $\cE'\subseteq\cE^0$, then it is the FCA, not a single fixed CAO, that is idempotent on $\cE'$. Here $1_{E}$ denotes the constant-one profile on the fiber $E$.

A property is \emph{cross-set} if it compares functionals assigned to different conditional sets. Such a property cannot be checked separately on each $G_E$; it requires compatibility relations between members of the family. Monotonicity with respect to set inclusion, consistency under zero extension, trace generation, and equivariance between different conditional sets are cross-set properties.

\subsection{Internality and Idempotency}
Fiberwise idempotency is closely related to the classical internality property of aggregation functions. For $u\in\bF_E$, write
$$\underline u_E:=\inf_{x\in E}u(x),\qquad
\overline u_E:=\sup_{x\in E}u(x).$$

\begin{proposition}\label{prop:internality-idempotency}
Let $E\in\cE^0$ and let $G_E\colon\bF_E\to[0,\infty]$ be nondecreasing and normalized by $G_E(0_E)=0$. Then $G_E$ is idempotent, see \eqref{eq:CAOidempotent},
if and only if
\begin{equation}\label{eq:internality}
\underline u_E\le G_E(u)\le \overline u_E
\end{equation}
for every $u\in\bF_E$.
\end{proposition}

\begin{proof}
Assume first that $G_E$ is idempotent. Since $\underline u_E 1_{E}\le u\le \overline u_E 1_{E},$ monotonicity gives
$$G_E(\underline u_E 1_{E})\le G_E(u)\le G_E(\overline u_E 1_{E}),$$ 
which is exactly \eqref{eq:internality}. Conversely, applying \eqref{eq:internality} to the constant profile $u=b 1_{E}$ yields $G_E(b 1_{E})=b$.
\end{proof}

If $G_E$ is homogeneous of degree one, i.e., $G_E(b u)=bG_E(u)$ for every $b\ge0$ and $u\in\bF_E$, then idempotency is equivalent to the single normalization condition $G_E(1_{E})=1.$ Indeed, homogeneity gives $G_E(b 1_{E})=bG_E(1_{E})$. Thus, in homogeneous models, the constant response is controlled by one scalar normalization condition on each fiber.

\subsection{Size-Dependent Idempotency}
In many conditional systems the response on constant profiles depends on the size, measure, or structural weight of the conditional set. Let
$s\colon\cE^0\to[0,\infty]$ be a size functional. Examples include cardinality, measure, capacity value, or an application-dependent complexity index. We say that an FCA has a \emph{size-dependent constant response} with respect to $s$ if there exists a function
$\eta\colon[0,\infty]\times[0,\infty)\to[0,\infty]$ such that
\begin{equation}\label{eq:size-response}
G_E(b 1_{E})=\eta(s(E),b)
\end{equation}
for all $E\in\cE^0$ and $b\ge0$.

\begin{proposition}\label{prop:size-idempotency}
Assume that an FCA satisfies \eqref{eq:size-response}. For $c\in s(\cE^0)$ put $$\cE_c:=\{E\in\cE^0:s(E)=c\}.$$ Then the FCA is idempotent on $\cE_c$ if and only if $\eta(c,b)=b$ with $b\ge0.$ It is fiberwise idempotent on all of $\cE^0$ if and only if $\eta(t,b)=b$ for all $t\in s(\cE^0)$ and all $b\ge0$.
\end{proposition}

\begin{proof}
This follows directly from the identity $G_E(b 1_{E})=\eta(s(E),b)$.
\end{proof}

\begin{example}
Let $X$ be finite and let $$G_E(u)=q(|E|)\sup_{x\in E}u(x),$$ where $q\colon\mN\to[0,\infty)$. Then
$G_E(b 1_{E})=q(|E|)b,$ so $\eta(t,b)=q(t)b$. The family is idempotent precisely on those cardinality levels $t$ for which $q(t)=1$. For instance, $q(t)=2/t$ yields idempotency exactly on two-element conditional sets.
\end{example}

Size-dependent idempotency illustrates why FCA-level language is needed. The property is not attached to a single fixed CAO, but to the assignment $E\mapsto G_E$ together with the structural class of conditional sets on which the constant response is required.

\subsection{Canonical Maps Between Fibers}
Let $F,E\in\cE^0$ with $F\subseteq E$. The restriction map is
\begin{equation*}
R_{E,F}\colon\bF_E\to\bF_F,\qquad R_{E,F} u=u|_F.
\end{equation*}
The zero-extension map is $Z_{F,E}\colon\bF_F\to\bF_E$, where
$$(Z_{F,E}v)(x)=
\begin{cases}
v(x),&x\in F,\\
0,&x\in E\setminus F.
\end{cases}
$$ These maps satisfy $R_{E,F}\circ Z_{F,E}=\id_{\bF_F}$. Ambient-space zero extensions are always written explicitly, so profiles on different fibers are never identified without a canonical map. If $H\subseteq F\subseteq E$, then
$$R_{E,H}=R_{F,H}\circ R_{E,F},\qquad
Z_{H,E}=Z_{F,E}\circ Z_{H,F}.$$

\begin{definition}\label{def:set-monotonicity-family}
An FCA represented by $\mathbf G=\{G_E:E\in\cE^0\}$ is said to be \emph{nondecreasing with respect to conditional sets} if
\begin{equation}\label{eq:set-nondecreasing}
G_F(R_{E,F}u)\le G_E(u)
\end{equation}
for all $F\subseteq E$ in $\cE^0$ and all $u\in\bF_E$. It is \emph{nonincreasing with respect to conditional sets} if the reverse inequality holds.
\end{definition}
In the original CAO notation, \eqref{eq:set-nondecreasing} becomes
$$\sA(f|F)\le \sA(f|E),\qquad F\subseteq E,$$ for every $f\in\bF$.

\begin{definition}\label{def:zero-extension-consistency}
An FCA represented by $\mathbf G=\{G_E:E\in\cE^0\}$ is \emph{consistent under zero extension} if
\begin{equation}\label{eq:zec}
G_E(Z_{F,E}v)=G_F(v)
\end{equation}
for every $F\subseteq E$ in $\cE^0$ and every $v\in\bF_F$.
\end{definition}
Equivalently, in CAO notation,
$$\sA(f\mI{F}|E)=\sA(f|F),\qquad F\subseteq E,$$ for every $f\in\bF$. Thus, adjoining new elements with a value of zero does not change the aggregated value.

\begin{proposition}\label{prop:zero-extension-implies-set-monotonicity}
Every FCA that is consistent under zero extension is nondecreasing with respect to conditional sets.
\end{proposition}

\begin{proof}
Let $F\subseteq E$ and $u\in\bF_E$. Since profiles are nonnegative,
$Z_{F,E}(R_{E,F}u)\le u$
pointwise on $E$. The monotonicity of $G_E$ together with \eqref{eq:zec} gives
$$G_F(R_{E,F}u)=G_E(Z_{F,E}(R_{E,F}u))\le G_E(u),$$
which completes the proof.
\end{proof}

The reverse implication does not hold. For instance, the mapping $G_E(u)=\sum\limits_{x\in E}u(x)+|E|\max\limits_{x\in E}u(x)$
defines an FCA that is nondecreasing with respect to conditional sets but is not consistent under zero extension.

Zero-extension consistency can be read as support-locality. If $u\in\bF_E$ is zero on $E\setminus F$, then $u=Z_{F,E}(u|_F)$ and hence $G_E(u)=G_F(u|_F)$. Thus, the family does not create new aggregation effects by adjoining zero-valued elements.

\subsection{Trace-Generated Families}
\begin{definition}\label{def:trace-generated-fca}
Let $\Phi\colon\bF\to[0,\infty]$ be nondecreasing and satisfy $\Phi(0_X)=0$. An FCA represented by $\mathbf G=\{G_E:E\in\cE^0\}$ is called \emph{trace-generated} by $\Phi$ if
\begin{equation*}
G_E(u)=\Phi(Z_{E,X}u)
\end{equation*}
for every $E\in\cE^0$ and $u\in\bF_E$.
\end{definition}

\begin{proposition}\label{prop:trace-generated-zero-extension}
Every trace-generated FCA is consistent under zero extension and, consequently, is nondecreasing with respect to conditional sets.
\end{proposition}

\begin{proof}
If $F\subseteq E$ and $v\in\bF_F$, then $Z_{E,X}(Z_{F,E}v)=Z_{F,X}v$. Hence, $$G_E(Z_{F,E}v)=\Phi(Z_{E,X}(Z_{F,E}v))=\Phi(Z_{F,X}v)=G_F(v).$$ The second assertion follows from Proposition~\ref{prop:zero-extension-implies-set-monotonicity}.
\end{proof}

\begin{example}
\label{ex:zec-nonrepresentable}
Let $X=\{1,2,3\}$ and consider the admissible context family
$\cE^0=\{\{1\},\{1,2\},\{2,3\}\}$. Define
\begin{align*}
G_{\{1\}}(u)&=u(1), &
G_{\{1,2\}}(u)&=u(1)+u(2),\\
G_{\{2,3\}}(u)&=u(2)u(3).&&
\end{align*}
The only nontrivial inclusion is $\{1\}\subseteq\{1,2\}$, and the corresponding zero-extension identity holds. Nevertheless, the ambient profile $(0,1,0)$ has two local representations:
\[
\begin{array}{c|c|c}
E & u\in\bF_E & G_E(u)\\ \hline
\{1,2\} & (0,1) & 1\\
\{2,3\} & (1,0) & 0.
\end{array}
\]
Thus, the local values do not define a single-valued functional on the trace domain, so no global trace generator exists. This example shows why compatibility only along nested contexts is weaker than global representability.
\end{example}

Trace-generated FCAs represent the case in which all local aggregations are obtained from one fixed global functional by zero extension. This is an important special case, but it does not exhaust the FCA framework. The local functional assigned to $E$ may depend on the conditional set itself through a measure, capacity, system of weights, normalization factor, interaction pattern, or structural constraint. Such families are genuinely conditional and need not be trace-generated.

\subsection{Equivariance}
Let $\Gamma$ be a group of bimeasurable bijections $\sigma\colon X\to X$ of $(X,\Sigma)$ such that $\sigma(E)\in\cE^0$ whenever $E\in\cE^0$. For $u\in\bF_E$, define the transported profile $u^\sigma\in\bF_{\sigma(E)}$ by
$$u^\sigma(y):=u(\sigma^{-1}(y)),\qquad y\in\sigma(E).$$

\begin{definition}
An FCA represented by $\mathbf G=\{G_E:E\in\cE^0\}$ is called $\Gamma$-\textit{equivariant} if $$G_{\sigma(E)}(u^\sigma)=G_E(u)$$
for every $E\in\cE^0$, every $u\in\bF_E$, and every $\sigma\in\Gamma$.
\end{definition}

Equivariance means that relabeling equivalent sources, criteria, or rules does not change the aggregated value. It is a family-level symmetry: it compares functionals assigned to different conditional sets.

\begin{proposition}\label{prop:trace-equivariance}
Assume that an FCA is trace-generated by $\Phi$ and that $\Phi$ is $\Gamma$-invariant, i.e., $$\Phi(f^\sigma)=\Phi(f),
\qquad f^\sigma(y)=f(\sigma^{-1}(y)).$$ Then the FCA is $\Gamma$-equivariant.
\end{proposition}

\begin{proof}
For $u\in\bF_E$, one has $Z_{\sigma(E),X}(u^\sigma)=(Z_{E,X}u)^\sigma.$ Therefore,
\begin{align*}
G_{\sigma(E)}(u^\sigma)
& = \Phi(Z_{\sigma(E),X}(u^\sigma)) = \Phi((Z_{E,X}u)^\sigma)
\\ & =\Phi(Z_{E,X}u)=G_E(u),
\end{align*}which gives the result.
\end{proof}

\subsection{Separating Examples}
The preceding implications are not reversible in general. The family $G_E(u)=\sup_{x\in E}u(x)$ is trace-generated by $\Phi(f)=\sup_{x\in X}f(x)$, hence zero-extension consistent and set-nondecreasing. If $\nu$ is a fixed nonnegative measure on $(X,\Sigma)$, then $G_E(u)=\int_Eu\md\nu$ is trace-generated by $\Phi(f)=\int_Xf\md\nu$. By contrast, the arithmetic mean on finite nonempty sets,
$$G_E(u)=\frac{1}{|E|}\sum_{x\in E}u(x),$$ is fiberwise idempotent, but not zero-extension consistent. If $F\subsetneq E$ and $v\not\equiv0$, then
$$G_E(Z_{F,E}v)=\frac{1}{|E|}\sum_{x\in F}v(x)\ne
\frac{1}{|F|}\sum_{x\in F}v(x)=G_F(v).$$ The family $G_E(u)=\inf_{x\in E}u(x)$ is fiberwise idempotent and set-nonincreasing, but not zero-extension consistent, since adding zero-valued elements may decrease the infimum to zero. The sum family
$$G_E(u)=\sum_{x\in E}u(x)$$ is zero-extension consistent and set-nondecreasing, but it is idempotent only on singleton sets. Finally, on $[0,1]$ the product family $$G_E(u)=\prod_{x\in E}u(x)$$ is set-nonincreasing, because every additional factor is at most one. On $[1,\infty)$ the same formula becomes set-nondecreasing. Thus, set monotonicity may depend not only on the formula but also on the admissible range of profiles. These examples show that fiberwise and cross-set properties are logically distinct.

\section{Global Representation of Fibered Families}
\label{sec:IV}

The canonical maps turn the local profile spaces into a fibered system. A local profile can be compared with a profile on another fiber only after both have been embedded into the ambient space. This observation yields a representation criterion that does not require the full universe $X$ to belong to $\cE^0$. We define the \emph{trace domain}
$$\mathscr T_{\cE}
:=
\bigcup_{E\in\cE^0} Z_{E,X}(\bF_E)
\subseteq\bF.
$$ Equivalently,
$$\mathscr T_{\cE} =
\bigl\{
f\in\bF:
\{x\in X:f(x)>0\}\subseteq E
\text{ for some }E\in\cE^0
\bigr\}.
$$ Thus, $\mathscr T_{\cE}$ consists precisely of the ambient profiles whose positive sets are contained in admissible nonempty contexts.

\begin{definition}
\label{def:global-order-compatibility}
An FCA $\mathbf G=\{G_E:E\in\cE^0\}$ is called \emph{globally order compatible} if
\begin{equation}
\label{eq:global-order-compatibility}
Z_{E,X}u\le Z_{F,X}v
\quad\Longrightarrow\quad
G_E(u)\le G_F(v)
\end{equation}
for all $E,F\in\cE^0$, $u\in\bF_E$, and $v\in\bF_F$.
\end{definition}

Condition~\eqref{eq:global-order-compatibility} simultaneously encodes consistency on common zero-extended profiles and monotonicity between profiles coming from different fibers, whenever their zero extensions are pointwise ordered. It is therefore stronger than checking monotonicity separately on each $G_E$.

\begin{theorem}[Fibered global representation]
\label{thm:fibered-global-representation}
For an FCA $\mathbf G=\{G_E:E\in\cE^0\}$, the following assertions are equivalent.
\begin{enumerate}
\item[(i)] The FCA is trace-generated by a nondecreasing functional $\Phi\colon\bF\to[0,\infty]$ satisfying $\Phi(0_X)=0$.
\item[(ii)] The FCA is globally order compatible.
\item[(iii)] The rule
\begin{equation}
\label{eq:partial-trace-functional}
\varphi(Z_{E,X}u):=G_E(u)
\end{equation}
defines a well-defined nondecreasing functional $\varphi\colon\mathscr T_{\cE}\to[0,\infty]$.
\end{enumerate}
If these conditions hold, define for each $f\in\mathbf F$,
\begin{align}
\underline\Phi(f)
&:=\sup\{\varphi(h):h\in\mathscr T_{\cE},\ h\le f\},
\label{eq:minimal-global-generator}\\
\overline\Phi(f)
&:=\inf\{\varphi(h):h\in\mathscr T_{\cE},\ f\le h\},\nonumber
\end{align} under the convention $\inf\emptyset:=\infty$. 
Then $\underline\Phi$ and $\overline\Phi$ are respectively the least and greatest nondecreasing trace generators of $\mathbf G$. In particular, every nondecreasing trace generator $\Phi$ satisfies
\begin{equation}
\label{eq:generator-interval}
\underline\Phi(f)\le\Phi(f)\le\overline\Phi(f),
\qquad f\in\bF,
\end{equation}
and the global generator is unique if and only if $\underline\Phi=\overline\Phi$.
\end{theorem}

\begin{proof}
If $\mathbf G$ is trace-generated by $\Phi$ and
$Z_{E,X}u\le Z_{F,X}v$, then
$$
G_E(u)=\Phi(Z_{E,X}u)\le\Phi(Z_{F,X}v)=G_F(v),
$$
so (i) implies (ii).

Assume (ii). If $Z_{E,X}u=Z_{F,X}v$, applying global order compatibility in both directions gives $G_E(u)=G_F(v)$; hence~\eqref{eq:partial-trace-functional} is well defined. If $h\le k$ in $\mathscr T_{\cE}$, choose representations $h=Z_{E,X}u$ and $k=Z_{F,X}v$. Then~\eqref{eq:global-order-compatibility} gives $\varphi(h)\le\varphi(k)$, proving (iii). Conversely, (iii) immediately implies (ii).

Assume (iii). The lower set in~\eqref{eq:minimal-global-generator} is nonempty because $0_X\in\mathscr T_{\cE}$ and $\varphi(0_X)=0$. Both $\underline\Phi$ and $\overline\Phi$ are nondecreasing. If $h\in\mathscr T_{\cE}$, monotonicity of $\varphi$ and the presence of $h$ in both defining families give
$\underline\Phi(h) = \varphi(h) = \overline\Phi(h).$ Thus, both functionals trace-generate $\mathbf G$, proving (i). Finally, any nondecreasing extension $\Phi$ of $\varphi$ satisfies
$\varphi(h)\le\Phi(f)$ whenever $h\le f$ and
$\Phi(f)\le\varphi(h)$ whenever $f\le h$. Taking the supremum and infimum yields~\eqref{eq:generator-interval}; the extremality and uniqueness assertions follow.
\end{proof}

\begin{theorem}[Representation hierarchy]
\label{thm:representation-hierarchy}
Let $\mathfrak A_{\cE}$ be the class of all FCAs on $\cE^0$, let
$\mathfrak Z_{\cE}$ be the subclass of zero-extension-consistent FCAs, and let
$\mathfrak T_{\cE}$ be the subclass of trace-generated FCAs. Then
\begin{equation*}
\mathfrak T_{\cE}\subseteq\mathfrak Z_{\cE}\subseteq\mathfrak A_{\cE}.
\end{equation*}
In general, each inclusion can be strict for a suitable context system. If $X\in\cE^0$, then
$\mathfrak T_{\cE}=\mathfrak Z_{\cE}$.
\end{theorem}

\begin{proof}
The first inclusion follows from Proposition~\ref{prop:trace-generated-zero-extension}; the second is immediate. On any context system containing two properly nested finite sets, the normalized arithmetic-mean family is an FCA that is not zero-extension consistent, so the second inclusion can be strict. Example~\ref{ex:zec-nonrepresentable} shows that the first inclusion can be strict: zero-extension consistency holds, but the induced rule on the trace domain is not even single valued, so Theorem~\ref{thm:fibered-global-representation} excludes a trace generator. If $X\in\cE^0$, zero-extension consistency gives
$G_E(u)=G_X(Z_{E,X}u)$, so every member of $\mathfrak Z_{\cE}$ is trace-generated.
\end{proof}

The hierarchy identifies three genuinely different levels: arbitrary local aggregation, compatibility under nested zero extensions, and global representability. The last level is governed by global order compatibility rather than by any property of the individual fibers alone.

\begin{example}\label{ex:extremal-trace-generators}
Let $X=\{1,2\}$, let $\cE^0=\{\{1\},\{2\}\}$, and define
$G_{\{i\}}(u)=u(i)$. The trace domain is the union of the two coordinate axes in $[0,\infty)^2$. Global order compatibility is immediate: two nonzero profiles lying on different axes are incomparable, while profiles lying on the same axis are ordered exactly as their corresponding local values. Hence the family is trace-generated.
For $f=(x,y)\in[0,\infty)^2$, Theorem~\ref{thm:fibered-global-representation}
gives
\begin{equation*}
\underline\Phi(x,y)=\max\{x,y\},
\end{equation*}
and
\begin{equation*}
\overline\Phi(x,y)=
\begin{cases}
x,&y=0,\\
y,&x=0,\\
\infty,&x>0\text{ and }y>0.
\end{cases}
\end{equation*}
Thus, the global representation is highly nonunique: both the
least generator
$\underline\Phi(x,y)=\max\{x,y\}$ and the monotone functional $\Phi(x,y)=x+y
$ generate the same local family.
The value $\overline\Phi(x,y)=\infty$ for $x,y>0$ reflects the fact that
no profile in the trace domain dominates a profile whose positive set meets both incomparable contexts. Consequently, the extremal generators describe not only the nonuniqueness of the global representation, but also the information that is absent because no admissible context contains both coordinates.
\end{example}

\section{Extension Theory for Partial Capacities}
\label{sec:V}

We now specialize to finite $X$ and take $\Sigma=2^X$. A \textit{normalized capacity} on $X$ is a monotone map
$\mu\colon2^X\to[0,1]$ satisfying $\mu(\emptyset)=0$ and $\mu(X)=1$, see~\cite{GrabischBook,WangKlir2009}. Local capacity data generally specify such a map only on part of the Boolean lattice.

\begin{definition}
Let $\mathcal D\subseteq2^X$ contain $\emptyset$ and $X$. A map
$\nu\colon\mathcal D\to[0,1]$ is a \emph{partial capacity} if
$\nu(\emptyset)=0$, $\nu(X)=1$, and
$$
A,B\in\mathcal D,\ A\subseteq B
\quad\Longrightarrow\quad
\nu(A)\le\nu(B).
$$
\end{definition}

\begin{theorem}[Extremal extension of a partial capacity]
\label{thm:partial-capacity-extension}
Every partial capacity $\nu\colon\mathcal D\to[0,1]$ on a finite set $X$ admits a normalized capacity extension to $2^X$. The least and greatest extensions are
\begin{align}
\nu_*(A)&:=\max\{\nu(B):B\in\mathcal D,\ B\subseteq A\},
\label{eq:partial-lower-extension}\\
\nu^*(A)&:=\min\{\nu(C):C\in\mathcal D,\ A\subseteq C\},
\label{eq:partial-upper-extension}
\end{align}
for $A\subseteq X$. Every normalized capacity $\mu$ extending $\nu$ satisfies
\begin{equation}
\label{eq:partial-extension-order}
\nu_*(A)\le\mu(A)\le\nu^*(A),
\qquad A\subseteq X.
\end{equation}
Consequently, the extension is unique if and only if $\nu_*=\nu^*$.
\end{theorem}

\begin{proof}
The defining families in~\eqref{eq:partial-lower-extension} and
\eqref{eq:partial-upper-extension} are nonempty because
$\emptyset,X\in\mathcal D$. If $B\subseteq A\subseteq C$ with
$B,C\in\mathcal D$, monotonicity of $\nu$ gives $\nu(B)\le\nu(C)$; hence
$\nu_*(A)\le\nu^*(A)$. Both functions are monotone in $A$, normalized at
$\emptyset$ and $X$, and agree with $\nu$ on $\mathcal D$. Thus, they are capacity extensions. If $\mu$ is any extension, then
$\nu(B)=\mu(B)\le\mu(A)$ for $B\subseteq A$ and
$\mu(A)\le\mu(C)=\nu(C)$ for $A\subseteq C$. Taking the maximum and minimum proves~\eqref{eq:partial-extension-order}.
\end{proof}

Let each $E\in\cE^0$ carry a normalized capacity
$\mu_E\colon2^E\to[0,1]$. The family is \emph{overlap compatible} if
\begin{equation*}
\mu_E(B)=\mu_F(B),
\qquad B\subseteq E\cap F,
\end{equation*}
for all $E,F\in\cE^0$.

\begin{corollary}[Gluing local capacities]
\label{cor:gluing-local-capacities}
A family $\{\mu_E:E\in\cE^0\}$ of local normalized capacities admits a global normalized capacity $\mu$ satisfying
$$
\mu(B)=\mu_E(B),
\qquad B\subseteq E,\ E\in\cE^0,
$$
if and only if it is overlap compatible. If this holds, put
$$
\mathcal D_0:=\bigcup_{E\in\cE^0}2^E,
\qquad
\mathcal D:=\mathcal D_0\cup\{X\},
$$
and define $\nu(B)$ as the common local value for $B\in\mathcal D_0$, with
$\nu(X)=1$ when $X\notin\mathcal D_0$. Then the least and greatest global completions are $\nu_*$ and $\nu^*$ from
Theorem~\ref{thm:partial-capacity-extension}.
\end{corollary}

\begin{proof}
Necessity follows by restricting a global capacity to a common coalition. Conversely, overlap compatibility makes $\nu$ well defined. If
$A,B\in\mathcal D_0$ and $A\subseteq B$, choose $E\in\cE^0$ with
$B\subseteq E$. Then also $A\subseteq E$, and overlap compatibility allows both partial values to be read from $\mu_E$; hence
$\nu(A)\le\nu(B)$. The case $B=X$ follows from $\nu(A)\le1$. Thus, $\nu$ is a partial capacity, and Theorem~\ref{thm:partial-capacity-extension} applies.
\end{proof}

This result separates nested compatibility from genuine gluing. Agreement only along inclusions controls zero-extension consistency; agreement on all overlaps is exactly what is needed for one global interaction model.

\section{Representation-Preserving Integral Schemes}
\label{sec:VI}

The universal-integral framework was introduced to provide a common setting for prominent nonadditive integrals and to separate the underlying monotone measure from the algebraic operation used in aggregation~\cite{KMP2010}. It belongs to a broader program of extending and axiomatizing fuzzy integrals, including interval-valued, copula-based, decomposition, and Choquet-like constructions~\cite{Bustince2013,GrecoRindone2013,WangZhangShen2024}. For the present fibered theory, the decisive issue is not the choice of one particular integral formula, but whether an integral respects canonical zero extensions of profiles and compatible extensions of its monotone set function. We isolate precisely this naturality property; our axioms are therefore not proposed as a replacement for the standard definition of a universal integral.

Throughout this section, the ambient profile class is restricted to the unit cube $[0,1]^X$, with local fibers $[0,1]^E$. An \emph{integral scheme} $\mathfrak U$ assigns to every finite $E$, every monotone set function
$\mu\colon2^E\to[0,1]$ with $\mu(\emptyset)=0$, and every
$u\in[0,1]^E$ a value $\mathfrak U_\mu(u)\in[0,1]$.

\begin{definition}
\label{def:representation-preserving-scheme}
An integral scheme $\mathfrak U$ is called \emph{calibrated and representation preserving} if:
\begin{enumerate}
\item[(U1)] for every $u,v\in[0,1]^E$, $u\le v$ implies $\mathfrak U_\mu(u)\le\mathfrak U_\mu(v)$, and
$\mathfrak U_\mu(0_E)=0$;
\item[(U2)] $\mathfrak U_\mu(\mI{B})=\mu(B)$ for every $B\subseteq E$;
\item[(U3)] whenever $F\subseteq E$ and monotone set functions
$\lambda\colon2^F\to[0,1]$ and $\mu\colon2^E\to[0,1]$, both vanishing at the empty set, satisfy $\mu|_{2^F}=\lambda$, then
\begin{equation*}
\mathfrak U_\mu(Z_{F,E}u)=\mathfrak U_\lambda(u),
\qquad u\in[0,1]^F.
\end{equation*}
\end{enumerate}
\end{definition}

Property (U3) is an embedding invariance: extending a profile by zero and extending its monotone set function without changing any value on the original fiber leave the integral unchanged. Property (U2) guarantees that the set function can be recovered from the integral on indicators; hence compatibility of integral values cannot conceal incompatible local set functions.

\begin{definition}
A family $\{\mu_E:E\in\cE^0\}$ is said to be \emph{projectively compatible} if
\begin{equation*}
\mu_E(B)=\mu_F(B)
\end{equation*}
for every $B\subseteq F\subseteq E$ with $F,E\in\cE^0$.
\end{definition}

\begin{theorem}[Projective integral representation]
\label{thm:projective-integral-representation}
Let $\mathfrak U$ be a calibrated representation-preserving scheme and define
$$
G_E^{\mathfrak U}(u):=\mathfrak U_{\mu_E}(u),
\qquad u\in[0,1]^E.
$$
Then $\mathbf G^{\mathfrak U}=\{G_E^{\mathfrak U}:E\in\cE^0\}$ is an FCA on the unit-cube profile class, and it is consistent under zero extension if and only if
$\{\mu_E:E\in\cE^0\}$ is projectively compatible.
\end{theorem}

\begin{proof}
Property (U1) gives the CAO conditions on every fiber. If the set functions are projectively compatible and $F\subseteq E$, then
$\mu_E|_{2^F}=\mu_F$, so (U3) yields
$$
G_E^{\mathfrak U}(Z_{F,E}u)=G_F^{\mathfrak U}(u).
$$
Conversely, apply zero-extension consistency to $u=\mI{B}$ with
$B\subseteq F\subseteq E$. By (U2), we get
$$
\mu_E(B)=\mathfrak U_{\mu_E}(Z_{F,E}\mI{B})
=\mathfrak U_{\mu_F}(\mI{B})=\mu_F(B). \qedhere
$$ 
\end{proof}

\subsection{Choquet and Universal-Integral Realizations}
The discrete Choquet scheme~\cite{Choquet1954,GrabischBook} satisfies Definition~\ref{def:representation-preserving-scheme}, since
\begin{equation*}
\Ch_\mu(u)=\int_0^1\mu(\{u\ge t\})\,\md t
\end{equation*}
and every positive level set is unchanged by zero extension.

A second canonical class comes directly from universal-integral framework \cite{KMP2010}. Let
$\otimes\colon[0,1]^2\to[0,1]$ be a semicopula, i.e., a nondecreasing operation with
$1\otimes a=a\otimes1=a$. Define the smallest $\otimes$-based universal integral by
\begin{equation}
\label{eq:semicopula-universal-integral}
\textrm{I}_{\otimes,\mu}(u)
:=\sup_{t\in[0,1]}t\otimes\mu(\{u\ge t\}).
\end{equation}
This scheme satisfies (U1)--(U3). Indeed, monotonicity and the neutral-element property imply $0\otimes a=0$ for every $a\in[0,1]$. Profile monotonicity then follows from the nesting of level sets, and the zero profile has integral value zero. For $u=\mI{B}$, the term corresponding to $t=0$ is zero, while for $0<t\le1$ the level set is $B$ and
$t\otimes\mu(B)\le 1\otimes\mu(B)=\mu(B)$, with equality at $t=1$; hence indicator calibration holds. Finally, (U3) follows because the positive level sets of $Z_{F,E}u$ coincide with those of $u$, whereas the terms at $t=0$ vanish on both fibers. The choices $\otimes=\min$ and $\otimes=\cdot$ give the Sugeno and Shilkret integrals, respectively~\cite{Sugeno1974,Shilkret1971,KMP2010}.

\begin{corollary}
\label{cor:classical-integrals-projective}
For Choquet-generated FCAs and for smallest semicopula-based universal-integral FCAs, zero-extension consistency is equivalent to projective compatibility of the local monotone set functions. In particular, this holds for the Choquet, Sugeno, and Shilkret integrals.
\end{corollary}

If all local set functions are normalized, projective compatibility is restrictive: for $F\subseteq E$ it forces
$\mu_E(F)=\mu_F(F)=1$. Thus, contextwise normalization may preserve fiberwise idempotency while destroying zero-extension consistency.

\begin{theorem}[Global integral representation]
\label{thm:global-integral-representation}
Let $\mathfrak U$ be a calibrated representation-preserving scheme and let
$\{\mu_E:E\in\cE^0\}$ be local normalized capacities. There exists a global normalized capacity $\mu$ such that
\begin{equation}
\label{eq:global-integral-representation}
\mathfrak U_\mu(Z_{E,X}u)=\mathfrak U_{\mu_E}(u)
\end{equation}
for all $E\in\cE^0$ and $u\in[0,1]^E$ if and only if the local capacities are overlap compatible.
\end{theorem}

\begin{proof}
If~\eqref{eq:global-integral-representation} holds, indicator calibration gives
$\mu_E(B)=\mu(B)=\mu_F(B)$ for every $B\subseteq E\cap F$. Conversely, Corollary~\ref{cor:gluing-local-capacities} provides a global capacity extending every $\mu_E$, and (U3) gives~\eqref{eq:global-integral-representation}.
\end{proof}

Thus, disagreement on a common coalition is an exact obstruction to every global representation based on a calibrated representation-preserving scheme; disagreement between nested contexts already obstructs zero-extension consistency.

We call $\mathfrak U$ \emph{capacity monotone} if
$\mu\le\lambda$ pointwise implies
$\mathfrak U_\mu(u)\le\mathfrak U_\lambda(u)$ for every profile $u$. The Choquet scheme and the semicopula-based schemes in~\eqref{eq:semicopula-universal-integral} have this property.

\begin{theorem}[Robust representation interval]
\label{thm:robust-representation-interval}
Assume that the local capacities are overlap compatible, and let
$\mathcal C_{\rm glob}$ be the set of their global normalized extensions. Let
$\mu_*$ and $\mu^*$ be the extremal completions from Theorem~\ref{thm:partial-capacity-extension}. If $\mathfrak U$ is capacity monotone, then for every $f\in[0,1]^X$,
\begin{equation*}
\mathfrak U_{\mu_*}(f)
=
\min_{\mu\in\mathcal C_{\rm glob}}\mathfrak U_\mu(f)
\le
\max_{\mu\in\mathcal C_{\rm glob}}\mathfrak U_\mu(f)
=
\mathfrak U_{\mu^*}(f).
\end{equation*}
Hence both endpoints are sharp, and the displayed interval is the smallest interval containing all globally represented scores compatible with the local capacity data.
\end{theorem}

\begin{proof}
Every $\mu\in\mathcal C_{\rm glob}$ satisfies
$\mu_*\le\mu\le\mu^*$ by
Theorem~\ref{thm:partial-capacity-extension}. Capacity monotonicity yields the two score inequalities. Since $\mu_*$ and $\mu^*$ themselves belong to
$\mathcal C_{\rm glob}$, the lower and upper bounds are attained.
\end{proof}

For the semicopula-based universal integral in
\eqref{eq:semicopula-universal-integral}, the sharp endpoints are obtained by
direct substitution:
\begin{equation*}
\sup_{t\in[0,1]}t\otimes\mu_*(\{f\ge t\})
\quad\text{and}\quad
\sup_{t\in[0,1]}t\otimes\mu^*(\{f\ge t\}).
\end{equation*}
Thus, the extension theorem produces computable robust bounds without selecting arbitrary values for unassessed coalitions.

\section{Representative Constructions}
\label{sec:VII}

The representation hierarchy gives a compact classification of familiar models. This viewpoint complements the usual operator-centered literature: recent fuzzy-system studies optimize or learn fuzzy measures, build Choquet-based decision procedures, and fuse heterogeneous fuzzy information~\cite{BeliakovWu2024,WangZhangShen2024,Qin2024,Gao2023}, whereas the results below classify the additional compatibility requirements created when the model domain itself varies. A fixed nondecreasing global functional $\Phi$ produces
$G_E(u)=\Phi(Z_{E,X}u)$ and therefore a globally representable FCA. This includes fixed-measure integrals, threshold-counting rules, and finite t-conorm accumulation. Normalized contextual means and context-dependent weights need not be globally representable. For
$$
G_E(u)=\sum_{x\in E}w_E(x)u(x),
$$
zero-extension consistency is equivalent to $w_E|_F=w_F$ whenever $F\subseteq E$. If every weight system is nonnegative and normalized, this forces the total weight assigned to $E\setminus F$ to be zero. Capacity-based families are governed instead by projective and overlap compatibility.

\begin{proposition}[Global representation of contextual weighted sums]
\label{prop:weighted-global-representation}
Assume that $X$ is finite and that
\begin{equation*}
G_E(u)=\sum_{x\in E}w_E(x)u(x),
\qquad w_E(x)\ge0.
\end{equation*}
The following assertions are equivalent:
\begin{enumerate}
\item[(i)] the FCA is trace-generated by a nonnegative linear functional on
$[0,\infty)^X$;
\item[(ii)] the FCA is trace-generated by some nondecreasing global functional;
\item[(iii)] the local weights agree on every overlap,
\begin{equation}
 w_E(x)=w_F(x),
 \qquad x\in E\cap F,
 \label{eq:weight-overlap-compatibility}
\end{equation}
for all $E,F\in\cE^0$.
\end{enumerate}
If these conditions hold, one may choose
$\Phi_w(f)=\sum_{x\in X}w(x)f(x)$, where $w(x)$ is the common local weight on
contexts containing $x$ and may be set to zero when no admissible context
contains $x$.
\end{proposition}

\begin{proof}
The implication (i)$\Rightarrow$(ii) is immediate. If (ii) holds and
$x\in E\cap F$, the singleton profiles on $E$ and $F$ have the same ambient zero extension. Hence
$w_E(x)=G_E(\mI{\{x\}})=G_F(\mI{\{x\}})=w_F(x)$, proving (iii).
Conversely, under~\eqref{eq:weight-overlap-compatibility} the global weight $w$ is well defined, nonnegative, and satisfies
$\Phi_w(Z_{E,X}u)=G_E(u)$ for every $E$ and $u$.
\end{proof}

This proposition is a concrete instance of the global representation theorem. For contextual weighted sums, compatibility along nested contexts characterizes
zero-extension consistency, whereas compatibility on arbitrary overlaps
characterizes the existence of a single global weighting model.

Two further standard classes included in the comparison below are finite
iterates of t-norms and t-conorms. Recall that a t-norm $T$ is a commutative and associative semicopula. A t-conorm $S$ is, dually, a commutative, associative, and nondecreasing operation on $[0,1]$ with neutral element $0$; see~\cite{KMPbook}. Associativity and commutativity make the finite iterates of t-norms and
t-conorms independent of the enumeration of a finite context.
Since $T(a,b)\le a$, a family generated by finite t-norm iterates is
set-nonincreasing on the unit-cube profile class. It generally fails
zero-extension consistency because $T(a,0)=0,$ so adjoining a zero-valued coordinate may annihilate the aggregate. By
contrast, the identity $S(a,0)=a$ implies that finite t-conorm iterates are zero-extension consistent and
set-nondecreasing. If $X$ is finite, the resulting t-conorm family is
trace-generated by the corresponding iterate over the entire universe. In
particular, the probabilistic sum
$S(a,b)=a+b-ab$ gives the evidence-fusion rule considered in Section~\ref{sec:VIII}. 

Table~\ref{tab:classification} summarizes the fiberwise, cross-context, and global representation properties of these constructions, together with the other representative families discussed above.

\begin{table}[t]
\centering
\caption{Typical behavior of representative FCA constructions.
ZEC denotes zero-extension consistency. The t-norm and t-conorm rows refer to the unit-cube profile class.}
\label{tab:classification}
\resizebox{\textwidth}{!}{%
\begin{tabular}{lcccc}
\toprule
Construction & Fiberwise idempotent & ZEC & Set behavior & Global representation \\
\midrule
$\sup_Eu$ & yes & yes & nondecreasing & yes \\
Fixed-measure integral & generally no & yes & nondecreasing & yes \\
Normalized contextual mean & yes & generally no & neither, generally & generally no \\
Context-dependent weighted sum & if normalized & iff nested-compatible & depends & iff overlap-compatible \\
Finite iterated t-norm & generally no & generally no & nonincreasing & generally no \\
Finite iterated t-conorm & generally no & yes & nondecreasing & yes \\
Choquet/Sugeno/Shilkret with local capacities & yes & iff projective & depends & iff overlap-compatible \\
Evidence rule~\eqref{eq:evidence-family-app} & generally no & yes & nondecreasing & yes \\
\bottomrule
\end{tabular}%
}
\end{table}

\section{Applications}
\label{sec:VIII}

We illustrate the framework through two applications for different purposes. Nonlinear fuzzy aggregation is routinely employed for information fusion, fault diagnosis, and multi-criteria ranking~\cite{Gao2023,WangZhangShen2024,Qin2024}. The first application explains the operational meaning of zero-extension consistency and the distinction between a trace-generated evidence rule and context-dependent interaction models. The second illustrates the extension and robust-representation theorems when several overlapping local capacity models are available but cross-context interactions are unspecified. The examples are deliberately transparent: they isolate the effect of a changing active set, which is obscured when it is combined immediately with learning, optimization, or imputation procedures.

\subsection{Context-Aware Fuzzy Evidence Fusion}
Let $X=\{T,V,P,A,L\}$
represent temperature, vibration, pressure, acoustic emission, and load. A fuzzy profile $u\colon X\to[0,1]$ measures the evidence supplied by each source. Possible active contexts include
\begin{align*}
E_{\rm steady}&=\{T,V,P,L\}, &
E_{\rm bearing}&=\{V,A,L\},\\
E_{\rm thermal}&=\{T,P,L\}.&&
\end{align*}
The main engineering requirements and their FCA formulations are summarized in Table~\ref{tab:requirements-properties}.

\begin{table}[t]
\centering
\caption{Operational requirements and corresponding FCA properties.}
\label{tab:requirements-properties}
\begin{tabular}{p{0.53\columnwidth}p{0.37\columnwidth}}
\toprule
Requirement & FCA property \\
\midrule
Larger evidence cannot lower the output & profile monotonicity \\
Adding positive evidence cannot lower alarm & set-nondecreasing behavior \\
A zero-evidence source has no effect & zero-extension consistency \\
Relabeling structurally equivalent sources has no effect & equivariance \\
Operating mode changes interactions & local capacities $\mu_E$ \\
\bottomrule
\end{tabular}
\end{table}

Let $\alpha_x\in[0,1]$ be source reliabilities and define
\begin{equation}
\label{eq:evidence-family-app}
G_E^{\rm ev}(u)=1-\prod_{x\in E}(1-\alpha_xu(x)).
\end{equation}
It is generated by the corresponding global probabilistic-sum functional; hence it is zero-extension consistent and set-nondecreasing. It is also equivariant under every relabeling that preserves the context family and the reliability assignment $x\mapsto\alpha_x$. Take
$$
(\alpha_T,\alpha_V,\alpha_P,\alpha_A,\alpha_L)
=(0.7,0.9,0.8,0.6,0.5)
$$
and on $E=\{T,V,P,L\}$ let
$$
u(T)=0.4,\quad u(V)=0.7,\quad u(P)=0.2,\quad u(L)=0.5.
$$
Then
$G_E^{\rm ev}(u) = 1-(0.72)(0.37)(0.84)(0.75) = 0.832168.
$
If $A$ becomes active with value zero, the output remains $0.832168$. If
$u(A)=0.45$, it increases to $0.87748264$.

More generally, let $y\notin E$, $a\in[0,1]$, and put
$u_a=Z_{E,E\cup\{y\}}u+a\mI{\{y\}}$. Then the marginal contribution of the
new source is exactly
\begin{align*}
G_{E\cup\{y\}}^{\rm ev}(u_a)-G_E^{\rm ev}(u) 
=\alpha_y a\bigl(1-G_E^{\rm ev}(u)\bigr).
\end{align*}
The formula simultaneously shows zero-extension consistency at $a=0$,
nonnegative response to positive evidence, and a saturation effect: the marginal impact of a new source decreases as the already accumulated evidence approaches one. For comparison, consider the normalized weighted mean
$$M_E(u)=\frac{\sum_{x\in E}\alpha_xu(x)}{\sum_{x\in E}\alpha_x}.$$
Its values are displayed in Table~\ref{tab:evidence-comparison}. Adding a zero-valued source decreases the score because the denominator changes. This is a legitimate contextual renormalization for average-performance evaluation, but it is a dilution effect in evidence accumulation.

\begin{table}[t]
\centering
\caption{Effect of changing the active source set.}
\label{tab:evidence-comparison}
\begin{tabular}{lccc}
\toprule
Rule & $E$ & $E\cup\{A\}$, $u(A)=0$ & $u(A)=0.45$ \\
\midrule
$G^{\rm ev}$ & 0.8322 & 0.8322 & 0.8775 \\
$M$ & 0.4552 & 0.3771 & 0.4543 \\
\bottomrule
\end{tabular}
\end{table}

The same system may require genuinely context-dependent interactions. On
$E_{\rm bearing}=\{V,A,L\}$, consider the normalized capacity
\begin{align*}
\mu(\{V\})&=0.35, & \mu(\{A\})&=0.25, & \mu(\{L\})&=0.20,\\
\mu(\{V,A\})&=0.70, & \mu(\{V,L\})&=0.75, & \mu(\{A,L\})&=0.55.
\end{align*}
It is monotone and assigns the pair $\{V,L\}$ a superadditive value,
$0.75>0.35+0.20$. For
$u(V)=0.7$, $u(A)=0.45$, and $u(L)=0.5$, we get
\begin{align*}
\Ch_\mu(u)
&=0.45+0.05\,\mu(\{V,L\})+0.20\,\mu(\{V\})=0.5575.
\end{align*}
If another operating context assigns a different value to the same common coalition, overlap compatibility fails. Theorem~\ref{thm:global-integral-representation} then proves that no single global calibrated fuzzy-integral model can reproduce both local interaction rules. Thus, the failure of global representation has a direct modeling interpretation: the interaction itself changes with the operating mode.

\subsection{Robust Transfer of Local Multi-Criteria Models}
Let $X=\{q,c,r,s\}$ represent quality, cost efficiency, reliability, and sustainability. Two expert panels provide normalized local capacities on the overlapping contexts $E_1=\{q,r,s\}$, $E_2=\{c,r,s\}$. Since
$E_1\cap E_2=\{r,s\}$, overlap compatibility requires the two capacities to agree on every coalition contained in $\{r,s\}$, namely on $\emptyset$, $\{r\}$, $\{s\}$, $\{r,s\}$. These are the values in the first four rows of Table~\ref{tab}. The remaining rows contain context-specific assessments and need not agree, since the corresponding coalitions do not belong to both local
domains. A dash indicates that the coalition is not contained in the corresponding context. Within each context, the displayed values are monotone with respect to set inclusion; therefore, both $\mu_{E_1}$ and $\mu_{E_2}$ are normalized capacities.

\begin{table}[t]
\centering
\caption{Complete specification of the local capacities. }
\label{tab}
\begin{tabular}{c@{\qquad}cc}
\toprule
Coalition $B$ & $\mu_{E_1}(B)$ & $\mu_{E_2}(B)$ \\ \midrule
$\emptyset$ & $0$ & $0$ \\
$\{r\}$ & $0.25$ & $0.25$ \\
$\{s\}$ & $0.20$ & $0.20$ \\
$\{r,s\}$ & $0.50$ & $0.50$ \\
\midrule
$\{q\}$ & $0.35$ & -- \\
$\{q,r\}$ & $0.65$ & -- \\
$\{q,s\}$ & $0.55$ & -- \\
$\{c\}$ & -- & $0.30$ \\
$\{c,r\}$ & -- & $0.60$ \\
$\{c,s\}$ & -- & $0.50$ \\
\midrule
$E_1=\{q,r,s\}$ & $1$ & -- \\
$E_2=\{c,r,s\}$ & -- & $1$ \\
\bottomrule
\end{tabular}
\end{table}


By Corollary~\ref{cor:gluing-local-capacities}, global capacity extensions exist. Interactions involving both $q$ and $c$ remain unspecified. For the profile
\begin{align*}
f(q)&=0.8, & f(c)&=0.6, &
f(r)&=0.4, & f(s)&=0.7.
\end{align*}
the distinct nonempty upper level sets entering the Choquet sum are
$X$, $\{q,c,s\}$, $\{q,s\}$, and $\{q\}$. The extremal capacity extensions satisfy
$$\mu_*(\{q,c,s\})=0.55,
\qquad
\mu^*(\{q,c,s\})=1.
$$
Consequently, Theorem~\ref{thm:robust-representation-interval}, applied to the Choquet scheme, gives
\begin{align*}
\Ch_{\mu_*}(f)
&=0.4+0.2(0.55)+0.1(0.55)+0.1(0.35)=0.6,\\
\Ch_{\mu^*}(f)
&=0.4+0.2(1)+0.1(0.55)+0.1(0.35)=0.69.
\end{align*}
Hence every global Choquet score compatible with both panels belongs to the sharp interval $[0.60,0.69]$, and both endpoints are attainable. Its width has the transparent decomposition
$$
0.69-0.60=0.2\,(1-0.55)=0.09.
$$
Here $0.2$ is the length of the level interval on which the unassessed coalition
$\{q,c,s\}$ is active, while $[0.55,1]$ is its admissible capacity range. The
score interval is therefore structural rather than statistical: it quantifies
exactly how uncertainty in a missing cross-context interaction propagates to the
final aggregate.

\section{Conclusion}
We developed a representation theory for fibered conditional aggregation. The restriction principle identifies the local functionals, while global order compatibility characterizes exactly when those local functionals admit one monotone ambient generator. The extremal generator formulas and the representation hierarchy show that zero-extension consistency and global representability are distinct unless the full universe is an admissible context. For finite systems, the partial-capacity extension theorem gives least and greatest global capacities and reduces the gluing of local capacities to overlap compatibility. Calibrated representation-preserving integral schemes unify the Choquet integral and semicopula-based universal integrals, including the Sugeno and Shilkret cases: projective compatibility characterizes zero-extension consistency, overlap compatibility characterizes global integral representation, and extremal extensions yield sharp robust score intervals. The applications show that these results formalize robustness to inactive sources, operating-mode-dependent interaction, and uncertainty caused by missing cross-context information. Future work will address data-driven estimation of fibered systems, dynamic contexts, and extension results beyond finite capacity domains.
\section*{Acknowledgments}
This work was supported by the Slovak Research and Development Agency under contract No.~APVV-21-0468. The second author also acknowledges the support of the internal scientific grant vvgs-2026-3897.

\end{document}